\documentclass{AIMS}
\usepackage{amsmath,amsfonts,amssymb,amsthm}
\usepackage{pgf,tikz,graphicx,epstopdf}

 \usepackage[colorlinks=true]{hyperref}

\hypersetup{urlcolor=blue, citecolor=red}

\def\currentvolume{X}
 \def\currentissue{X}
  \def\currentyear{200X}
   \def\currentmonth{XX}
    \def\ppages{X--XX}
     \def\DOI{10.3934/xx.xx.xx.xx}

\newtheorem{theorem}{Theorem}[section]
\newtheorem{corollary}{Corollary}

\theoremstyle{definition}
\newtheorem{definition}[theorem]{Definition}
\newtheorem{remark}{Remark}

\newtheorem{example}{Example}

\renewcommand{\le}{\leqslant}
\renewcommand{\ge}{\geqslant}

\usetikzlibrary{arrows}

\definecolor{eqeqeq}{rgb}{0.8784313725490196,0.8784313725490196,0.8784313725490196}
\definecolor{yqyqyq}{rgb}{0.5019607843137255,0.5019607843137255,0.5019607843137255}
\definecolor{wqwqwq}{rgb}{0.3764705882352941,0.3764705882352941,0.3764705882352941}
\definecolor{aqaqaq}{rgb}{0.6274509803921569,0.6274509803921569,0.6274509803921569}
\definecolor{cqcqcq}{rgb}{0.7529411764705882,0.7529411764705882,0.7529411764705882}

\renewcommand{\theenumi}{\roman{enumi}}
\renewcommand{\labelenumi}{(\theenumi)}

\newcommand{\id}{\mathrm{Id}}

\newcommand{\ass}{\textrm{Ass}}

\newcommand{\vol}{\operatorname{vol}}

\newcommand{\fix}{\operatorname{Fix}}
\newcommand{\aut}{\operatorname{Aut}}

\renewcommand{\geq}{\geqslant}

\newcommand{\ent}{\mathsf{h}}
\newcommand{\relent}{\mathsf{h}_{\mathsf{rel}}}

\newcommand{\logmahler}{\mathsf{m}}
\newcommand{\laurent}{\mathbb{Z}[u_1^{\pm 1},\dots,u_d^{\pm 1}]}
\newcommand{\chull}{\mathcal{H}}

\newcommand{\wsyn}{S}
\newcommand{\ssyn}{S^\star}
\newcommand{\wrad}{r}
\newcommand{\srad}{r^\star}
\newcommand{\smeas}{\nu^\star}

\title[Directional uniformities, periodic points, and entropy]{Directional uniformities, periodic points, and entropy}
\author[Richard Miles and Thomas Ward]{}


\email{richard.miles@math.uu.se}
\email{t.b.ward@durham.ac.uk}

\begin{document}
\maketitle

\centerline{\scshape Richard Miles}
\medskip
{\footnotesize
\centerline{Uppsala Universitet}
 \centerline{L{\"a}gerhyddsv{\"a}gen 1, Hus 1, 5 och 7}
   \centerline{75106 Uppsala, Sweden}
}

\medskip

\centerline{\scshape Thomas Ward}
\medskip
{\footnotesize
   \centerline{Durham University}
   \centerline{Durham DH1 3LE, UK}
}

\bigskip

\centerline{(Communicated by the associate editor name)}

\begin{abstract}
Dynamical systems generated by~$d\ge2$
commuting homeomorphisms (topological~$\mathbb{Z}^d$-actions)
contain within them
structures on many scales,
and in particular contain
many actions of~$\mathbb{Z}^k$ for~$1\le k\le d$.
Familiar dynamical invariants for
homeomorphisms, like entropy and
periodic point data, become more complex
and permit multiple definitions. We briefly
survey some of these and other related
invariants in the setting of
algebraic~$\mathbb{Z}^d$-actions,
showing how, even in
settings where
the natural entropy as a~$\mathbb{Z}^d$-action
vanishes, a powerful theory of
directional entropy
and periodic points
can be built. An underlying theme
is uniformity in dynamical
invariants as the direction changes,
and the connection between this
theory and problems in number theory;
we explore this for several
invariants.
We also highlight Fried's
notion of average entropy and its
connection to uniformities in growth
properties, and prove a new relationship
between this entropy and periodic point
growth in this setting.
\end{abstract}

\maketitle

\section{Introduction}

A topological dynamical system general by~$d\ge2$
commuting homeomorphisms
may be studied on multiple
scales. At one end it may be viewed as a single
action of~$\mathbb Z^d$, with global invariants for the action
like the joint topological entropy playing a central role, and
with well-known rigidity phenomena arising.
At the
opposite extreme it may be studied one direction at a
time, which essentially reduces to asking for
the relationship between properties
of one homeomorphism and properties of other
commuting ones. In attempting to navigate between these
two extremes, the first of which can lose the rich infomation
and structures present inside an abelian group action,
and the second of which loses touch with the
novel phenomena
that can only arise for higher-dimensional acting
groups, two thematic approaches arise naturally.

The first is \emph{subdynamics},
introduced by Boyle and Lind,
in which dynamical properties are
formulated for~$k$-dimensional subspaces of~$\mathbb R^d$,
and the whole action is studied via properties of its
subactions. Two particular features of this approach are:
\begin{enumerate}
\item the natural emergence of critical dimensions for
properties like entropy and expansiveness, leading
to notions of entropy rank and expansive rank;
\item a `subdynamics philosophy', that dynamical properties
of~$k$-dimensional subactions should be constant,
or should vary in a smooth or predictable fashion,
while the~$k$-space moves within connected components
of the subset of the Grassmanian manifold of~$k$-spaces in~$\mathbb R^d$
comprising the expansive~$k$-spaces.
\end{enumerate}
We refer to the work of Boyle and Lind~\cite{MR1355295} for
a detailed account of subdynamics,
and that of Einsiedler, Lind and the authors~\cite{MR1869066}
for an account in the algebraic setting.

The second is \emph{directional dynamics and uniformities}, in which
the philosophy of subdynamics is brought to bear on
specific questions in single directions as they vary. Here
two particular focal points arise:
\begin{enumerate}
\item uniformities, in which
the parameters describing dynamical properties are found to
have no `singularities' as the direction is changed, even across
non-expansive directions;
\item in contrast,
inverse subdynamics results, where non-expansive
directions are detected by changes in directional
dynamical properties.
\end{enumerate}

Our main focus is on algebraic~$\mathbb{Z}^d$-actions where every
element of the action has finite entropy, though essentially all the
questions we discuss have natural formulations for~$k$-dimensional
subactions in~$\mathbb Z^d$-actions with~$k$-dimensional subactions
of finite positive entropy. Roughly speaking, the questions
then become notationally and technically a little more difficult,
but the real obstacle to progress in this direction is that the
Diophantine results underpinning the methods are no longer available.
The actions we will be discussing in the main are
called \emph{entropy rank one actions} in the language
of Einsiedler and Lind~\cite{MR2031042},
whose work provides a
structure theorem for such actions.
Where we refer to actions of higher rank,
the reader should consult the papers of Boyle and Lind~\cite{MR1355295}
or Einsiedler, Lind and the authors~\cite{MR1869066} for more
information.

To help orient the reader, we briefly list some
standard
examples of the
type of systems we will be discussing.
\begin{enumerate}
\item The commuting endomorphisms~$x\longmapsto2x$
and~$x\longmapsto3x$ modulo~$1$ on the additive circle~$\mathbb{T}
=\mathbb{R}/\mathbb{Z}$ together generate an action of~$\mathbb{N}^2$.
The natural extension of this system is then a~$\mathbb{Z}^2$
action by continuous automorphisms of a compact metrizable
group. The action is mixing, essentially because of
the arithmetic independence condition
\[
2^i3^j=1\Longrightarrow i=j=0.
\]
The action has entropy rank one because each individual map
(corresponding to multiplication by the integer~$2^a3^b$
if~$a,b>0$) has finite
entropy, positive if~$(a,b)\neq(0,0)$.
A property that lies much deeper
is that the~$S$-unit theorem implies that
the system is also mixing of all orders,
as shown by Schmidt and the second author~\cite{MR1193598}.
\item If~$A,B\in{\rm{GL}}_3(\mathbb{Z})$ are commuting matrices with
the property that
\[
A^iB^j=I\Longrightarrow i=j=0
\]
then they together generate a~$\mathbb{Z}^2$-action~$\alpha$
by automorphisms
of the torus~$\mathbb{T}^3$, with the
action of~$\alpha^{(a,b)}$ being multiplication by~$A^aB^b$.
This will be a mixing action of
entropy rank one, and once again will be mixing of
all orders for the same Diophantine reason.
\item Higher entropy rank may be seen in the following
example.
Let
\[
X=\{x=(x_n)\mid x\in\mathbb{T}^3, n\in\mathbb{Z}\}=\left({\mathbb{T}^3}\right)^{\mathbb{Z}}
\]
and choose~$A,B$ as above.
Then we may define an action~$\alpha$ of~$\mathbb{Z}^3$ on~$X$
by defining
\[
\left(\alpha^{(a,b,c)}(x)\right)_k=A^aB^bx_{k+c}.
\]
That is, the first two coordinates~$(a,b)$ act on each
coordinate in the shift space as a toral automorphism,
and the third coordinate~$c$ acts as the shift on the
sequence space~$X$ with alphabet~$\mathbb{T}^3$.
Any single direction (represented by a point~$(a,b,c)\in\mathbb{Z}^3
\setminus\{(0,0,0)\}$) has infinite entropy
(indeed, any automorphism of
positive entropy commuting with a~$\mathbb{Z}^2$
action of completely positive entropy is
forced to have infinite
entropy by a result of Morris and the second
author~\cite{MR1656849}),
so the entropy rank is strictly greater than one.
This system is mixing, and has entropy rank two.
\item The natural shift action of~$\mathbb{Z}^2$ on the
compact metrizable group
\[
\{x=(x_{n,m})\in\{0,1\}^{\mathbb{Z}^2}\mid x_{n,m}+x_{n+1,m}+x_{n,m+1}=0\mbox{ mod }2
{\ }\forall{\ }m,n\in\mathbb Z\}
\]
is a~$\mathbb{Z}^2$-action by automorphisms of a compact
totally disconnected
group.
This also is mixing and has entropy rank one, but is not mixing
of all orders.
\end{enumerate}
For a general overview of systems of this sort, we refer to
the monograph of Schmidt~\cite{MR1345152}.

\section{Entropies and Lyapunov lists}\label{lyupunov_lists}

For a~$\mathbb{Z}^d$-action~$\alpha$ by automorphisms of a compact abelian group~$X$, the notion of entropy developed by
Milnor~\cite{MR955558} may be used to define a directional
entropy for any~$\mathbf{t}\in\mathbb{R}^d$. We denote this simply as~$\ent(\mathbf{t})$, as dependence on the action will always be clear. Whenever~$\mathbf{t}\in\mathbb{Z}^d$,
the value~$\ent(\mathbf{t})$ coincides with the topological entropy of
the single homeomorphism~$\alpha^\mathbf{t}$. For other values of~$\mathbf{t}$, the definition is a little technical, corresponding to the more general concept
of the entropy of
a~$k$-frame in the work of Milnor~\cite{MR955558}
(see also Boyle and Lind~\cite{MR1355295}),
with~$k=1$.
Instead of going back to the definition, we refer to a convenient formula for~$\ent(\mathbf{t})$,~$\mathbf{t}\in\mathbb{R}^d$, which is
analogous to the directional entropy formula for a smooth action of~$\mathbb{Z}^d\times\mathbb{R}^n$ on a compact manifold in terms of Lyapunov exponents. When~$\alpha$ has entropy rank one, the formula reveals the \emph{directional entropy function}~$\ent:\mathbb{R}^d\rightarrow\mathbb{R}_{\geqslant 0}$ to be a continuous piecewise linear function, and also a semi-norm on~$\mathbb{R}^d$. Although we are only concerned with entropy rank one actions here,
we note that Milnor's general framework provides a subtle insight into the structure of algebraic~$\mathbb{Z}^d$-actions where individual elements of the action may have infinite entropy,
described in the work of Einsiedler, Lind and the authors~\cite{MR1869066}.

To state the directional entropy formula, some preliminaries are required. The \emph{descending chain condition} on closed~$\alpha$-invariant subgroups of~$X$ is the condition that any chain of such subgroups
\[
X=X_0\supset X_1\supset X_2\supset\cdots
\]
stabilizes. When~$\alpha$ has entropy rank one,
this condition ensures that the dynamical system~$(X,\alpha)$
has a finite Lyapunov list, in the sense of Einsiedler
and Lind~\cite{MR2031042}, which we will proceed to describe.
Following Kitchens and Schmidt~\cite{MR1036904}, we
recall the standard correspondence between~$(X,\alpha)$ and a module~$M$ over the ring of Laurent polynomials~$R_d=\laurent$, obtained by identifying the application of each dual automorphism~$\widehat{\alpha}^\mathbf{n}$ with multiplication by~$u^\mathbf{n}=u_1^{n_1}\cdots u_d^{n_d}$ and extending this in a natural way to polynomials. The descending chain condition
is equivalent to the assumption that~$M$ is a Noetherian module,
in which case it will have a finite set of associated prime ideals~$\ass(M)$. Furthermore,
under the assumption that~$\alpha$ has entropy rank
one, Einsiedler and Lind~\cite[Prop.~6.1]{MR2031042}
showed that for each~$\mathfrak{p}\in\ass(M)$
that is not a maximal ideal, the domain~$R_d/\mathfrak{p}$ has Krull dimension~1,
meaning that the field of fractions
of~$R_d/\mathfrak{p}$ is a global field,
which we denote~$\mathbb{K}(\mathfrak{p})$.
Further background on global fields may be found in
the monograph of Cohn~\cite{MR1140705}, for example. There is a set of places~$\mathcal{P}(\mathbb{K}(\mathfrak{p}))$ corresponding to inequivalent absolute values defined on~$\mathbb{K}(\mathfrak{p})$, and we let~$|\cdot|_v$ denote the normalised absolute value corresponding to the place~$v$.
Set
\[
\mathcal{S}(\mathfrak{p})=\{v\in\mathcal{P}(\mathbb{K}(\mathfrak{p}))\mid|R_d/\mathfrak{p}|_v\mbox{ is an unbounded subset of }\mathbb{R}\}.
\]
and note that this is a finite set because~$R_d/\mathfrak{p}$ is finitely generated. Define the \emph{Lyapunov list} for the prime~$\mathfrak{p}$ by
\[
\mathcal{L}(\mathfrak{p})=\{(\log|\overline{u}_1|_v,\dots,\log|\overline{u}_d|_v)\mid v\in \mathcal{S}(\mathfrak{p})\},
\]
where~$\overline{u}_i$ denotes the image of~$u_i$ in the domain~$R_d/\mathfrak{p}$. To define the \emph{Lyapunov list}~$\mathcal{L}(\alpha)$ for~$(X,\alpha)$, we form the union of Lyapunov lists over all
those associated primes~$\mathfrak{p}\in\ass(M)$ that are not maximal ideals, with each~$\mathcal{L}(\mathfrak{p})$ repeated the same number of times as the dimension of the~$\mathbb{K}(\mathfrak{p})$-vector space~$M\otimes\mathbb{K}(\mathfrak{p})$. In this way, the Lyapunov list reflects exactly the list of domains of the form~$R_d/\mathfrak{p}$ with Krull dimension~1 that appear as factors in any prime filtration of the module~$M$
by Einsiedler and Lind~\cite[Lem.~8.2]{MR2031042}, and the only other possible factors in such a filtration are of the form~$R_d/\mathfrak{m}$, where~$\mathfrak{m}$ is a maximal ideal. Furthermore, if~$\alpha$
is also assumed to be mixing, then
if any such maximal ideals appear then they cannot be associated
primes by work of the first author~\cite[Lem.~3]{MR2308145}. For any~$\mathbf{n}\in\mathbb{R}^d$, the directional entropy is given by
Abramov's formula
\begin{equation}\label{directional_entropy_formula}
\ent(\mathbf{n})=\sum_{\mathbf{\boldsymbol{\ell}}\in\mathcal{L}(\alpha)}\max\{\boldsymbol{\ell}\cdot\mathbf{n},0\}.
\end{equation}
For further discussion and the proof of~\eqref{directional_entropy_formula}
in this context see
Einsiedler and Lind~\cite[Sec.~8]{MR2031042}; the original
calculation is essentially due to Abramov~\cite{MR0117322}
and a modern geometrical proof may be found in
a treatment by Lind and the second author~\cite{MR961739}.

Einsiedler and Lind applied their methods to give a similar formula for fibre entropies related to skew product transformations, and subsequently to give a formula for the relational entropy introduced by
Friedland~\cite{MR1411226}, which may be defined for any closed subset~$\mathcal{R}\subset X\times X$, as follows. Set
\[
\mathcal{X}(\mathcal{R})=\{x\in X^\mathbb{N}\mid(x_i,x_{i+1})\in\mathcal{R}\mbox{ for all }i\in\mathbb{N}\},
\]
and let~$\sigma_{\mathcal{X}(\mathcal{R})}$ denote the one-sided shift on~$\mathcal{X}(\mathcal{R})$. The \emph{relational entropy}~$\relent(\mathcal{R})$ of~$\mathcal{R}$ is the topological entropy of~$\sigma_{\mathcal{X}(\mathcal{R})}$. For example, if~$\mathcal{R}_A$ is the graph of a single automorphism~$A:X\rightarrow X$, then~$\relent(\mathcal{R}_A)$ is the topological entropy of~$A$.
Geller and Pollicott~\cite{MR1636888} considered relational entropy for the union of the graphs of two commuting transformations, and they
conjectured a formula for~$\relent(\mathcal{R}_A\cup\mathcal{R}_B)$ when~$A$ and~$B$ are commuting automorphisms of a torus~$\mathbb{T}^m$. This was corrected by Einsiedler and Lind, who show~\cite[Th.~2.4]{MR2031042} that, under the independence assumption
that~$\{x\in X\mid Ax=Bx\}$ has Haar measure zero,
\[
\relent(\mathcal{R}_A\cup\mathcal{R}_B)=
\max_{E\subset\{1,\dots,m\}}
\log\left(
\prod_{j\in E}e^{s_j}
+
\prod_{j\in E}e^{t_j}
\right)
,
\]
where~$(s_1,t_1),\dots,(s_m,t_m)$ are the Lyupunov vectors of the~$\mathbb{Z}^2$-action generated by~$A$ and~$B$.
This is also in fact the topological entropy of the skew product of~$A$ and~$B$ over the two shift, and Einsiedler and Lind provide more general formul\ae{\ }for relations using skew products and topological pressures.

Another notion of entropy, that we will show to
be closely related to the uniform growth of
periodic points for the systems we
study, was introduced by Fried~\cite{MR677244} for an action of a compactly generated Lie group on a probability space. This entropy is defined in our setting with Haar measure as the measure on~$X$, so that the corresponding measure theoretic entropy coincides with topological entropy. Firstly, note that the directional entropy
function~$\ent:\mathbb{R}^d\rightarrow\mathbb{R}_{\geqslant 0}$ is
clearly measurable by construction.
Fried's original definition includes the volume of the generalized
octahedron~$\{(t_i)\in\mathbb{R}^d\mid\sum_{i=1}^d|t_i|<1\}$,
which is~$\frac{2^d}{d!}$. For ease, we accommodate this preliminary calculation into the following.

\begin{definition}
Suppose that~$\alpha$ is an entropy rank one action with directional entropy function~$\ent:\mathbb{R}^d\rightarrow\mathbb{R}_{\geqslant 0}$ and let~$U(\alpha)=\{\mathbf{t}\in\mathbb{R}^d\mid\ent(\mathbf{t})\leqslant 1\}$.
The \emph{Fried average entropy} of~$\alpha$ is
\[
h^*(\alpha)=\frac{2^d}{d!\vol(U(\alpha))}.
\]
\end{definition}

In Section~\ref{bounds_and_uniformities}, we will see that
if~$\alpha$ is assumed to be mixing, then~$\ent$ is bounded away from zero on the unit sphere
\[
\mathbb{S}^{d-1}=\{\mathbf{t}\in\mathbb{R}^d\mid\Vert\mathbf{t}\Vert=1\},
\]
and it therefore
defines a norm on~$\mathbb{R}^d$
(throughout,~$\Vert\cdot\Vert$ will denote the Euclidean norm, and we continue to use the notation~$\ent(\cdot)$ when the directional entropy is a norm).
It then follows that~$U(\alpha)$ is the closed unit ball in this norm, and~$U(\alpha)$ is a convex polytope in~$\mathbb{R}^d$. In the setting of smooth~$\mathbb{Z}^d\times\mathbb{R}^n$-actions on compact manifolds,
A.~Katok, S.~Katok and Hertz~\cite{katok_fried_entopy} describe such details, together with numerous other aspects of Fried entropy, including consequences of the directional entropy formula (or \emph{Pesin entropy formula}, as it is
a special case of the more general
results for diffeomorphisms in the case that~$X$
is a manifold), lower bounds for~$h^*$, and a relationship with slow entropy and regulators in algebraic number fields.

\section{Directional entropy and Fried average entropy examples}

This section is devoted to visualising the unit ball~$U(\alpha)$ in the directional entropy norm and calculating the Fried average entropy for some well-known algebraic examples.

\begin{example}[The~$\times 2$,~$\times 3$ example]\label{times2_times3_example}
The compact group dual to the ring~$\mathbb{Z}[\frac{1}{6}]$ is a one dimensional solenoid~$X$. Via duality, both of the maps~$x\longmapsto 2x$
and~$x\longmapsto 3x$ are invertible and define automorphisms of~$X$ that generate a mixing~$\mathbb{Z}^2$-action~$\alpha$. Furthermore, the group~$X$ is locally
isometric to a product of an interval in~$\mathbb{R}$,
an open set in the~$2$-adic
numbers~$\mathbb{Q}_2$, and
an open set in
the~$3$-adic numbers~$\mathbb{Q}_3$.
As described
by Einsiedler and Lind~\cite[Ex.~6.7]{MR2031042},
the Lyapunov list for this
example is given by the prime ideal~$\mathfrak{p}=(u_1-2,u_2-3)\subset R_2$ and the set of places~$\{2,3,\infty\}$
in~$\mathcal{P}(\mathbb{Q})$, so that
\[
\mathcal{L}(\alpha)=
\langle
(\log|2|_2,\log|3|_2), (\log|2|_3,\log|3|_3), (\log|2|_\infty,\log|3|_\infty)
\rangle.
\]
The~$x$- and~$y$-axes, together with the line~$x\log 2+y\log 3=0$ divide the plane into 6 regions, as shown in Figure~\ref{times2_times3_diagram}. Starting in the first quadrant and moving anticlockwise, label these~$\mathcal{C}_1,\dots,\mathcal{C}_6$. Then, for any~$(x,y)\in\mathbb{R}^2$,
\[
\ent((x,y))=
\left\{
\begin{array}{ll}
|x|\log 2 + |y|\log 3 & \mbox{ for } (x,y)\in\mathcal{C}_1\cup\mathcal{C}_4;\\
|y|\log 3 & \mbox{ for } (x,y)\in\mathcal{C}_2\cup\mathcal{C}_5;\\
|x|\log 2 & \mbox{ for } (x,y)\in\mathcal{C}_3\cup\mathcal{C}_6.\\
\end{array}
\right.
\]
This formula is related to the expansive subdynamics of this system
described by Boyle and Lind~\cite[Ex.~6.5]{MR1355295}.
The directional entropy unit ball~$U(\alpha)$ is shown in Figure~\ref{times2_times3_diagram}, and we find
\[
\vol(U(\alpha))=\frac{3}{(\log 2)(\log 3)}
\]
so
\[
h^*(\alpha)=\frac{6}{(\log 2)(\log 3)}.
\]

\begin{center}
\begin{figure}[h]
\begin{tikzpicture}[line cap=round,line join=round,>=triangle 45,x=1.4cm,y=1.4cm]
\draw[->,color=black] (-2.731,0.0) -- (2.496,0.0);
\draw[->,color=black] (0.0,-1.449) -- (0.0,1.540);
\fill[line width=1.6pt,fill=black,fill opacity=0.1] (-1.44,0.0) -- (-1.43,0.91) -- (0.0,0.91) -- (1.44,0.0) -- (1.44,-0.91) -- (0.0,-0.91) -- cycle;
\draw [line width=1.0pt] (-1.44,0.0)-- (-1.43,0.91);
\draw [line width=1.0pt] (-1.43,0.91)-- (0.0,0.91);
\draw [line width=1.0pt] (0.0,0.91)-- (1.44,0.0);
\draw [line width=1.0pt] (1.44,0.0)-- (1.44,-0.91);
\draw [line width=1.0pt] (1.44,-0.91)-- (0.0,-0.91);
\draw [line width=1.0pt] (0.0,-0.91)-- (-1.44,0.0);
\draw (1.4,0.55) node[anchor=north west] {$\frac{1}{\log(2)}$};
\draw (-2.45,0.55) node[anchor=north west] {$-\frac{1}{\log(2)}$};
\draw (0.06,1.32) node[anchor=north west] {$\frac{1}{\log(3)}$};
\draw (0,-0.89) node[anchor=north west] {$-\frac{1}{\log(3)}$};
\draw [line width=1.0pt,dash pattern=on 2pt off 2pt] (-1.77,1.13)-- (1.79,-1.13);
\draw (2.45,0) node[anchor=west] {$x$};
\draw (0,1.47) node[anchor=south] {$y$};
\begin{scriptsize}
\draw [fill=black] (-1.44,0.0) circle (1.5pt);
\draw [fill=black] (0.0,0.91) circle (1.5pt);
\draw [fill=black] (1.44,0.0) circle (1.5pt);
\draw [fill=black] (0.0,-0.91) circle (1.5pt);
\end{scriptsize}
\end{tikzpicture}
\caption{The directional entropy unit ball for~$\times 2$~$\times 3$.\label{times2_times3_diagram}}
\end{figure}
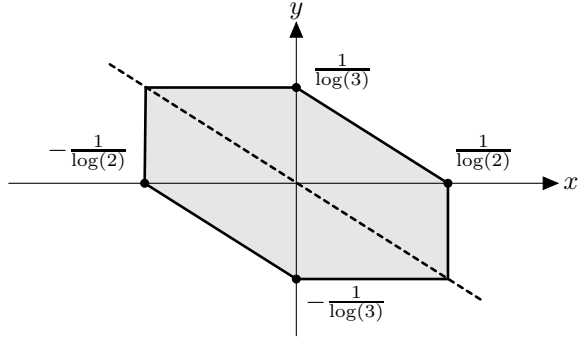
\end{center}
\end{example}

\begin{example}[Ledrappier's example]\label{ledrappiers_example}
The following example is one of a type
introduced by Ledrappier~\cite{MR512106} to
show that mixing does not imply higher-order
mixing for~$\mathbb Z^2$-actions with countable
Lebesgue spectrum.
The restriction of the full~$\mathbb{Z}^2$-shift to the compact abelian group
\[
\{x=(x_{i,j})\in\mathbb{F}_2^{\mathbb{Z}^2}\mid
x_{i,j}+x_{i+1,j}+x_{i,j+1}=0\mbox{ for all }(i,j)\in\mathbb{Z}^2\}
\]
is mixing, yet is not mixing of higher orders.
Just as in
the work of Boyle and Lind~\cite[Exs.~6.8, 8.1]{MR2031042},
Ledrappier's example corresponds to the~$R_2$-module~$R_2/\mathfrak{p}$, where~$\mathfrak{p}=(2, 1+u_1+u_2)$, and
\[
R_2/\mathfrak{p}\cong\mathbb{F}_2[t^{\pm 1},(1+t)^{\pm 1}],
\]
for an indeterminate~$t$, where the isomorphism is given by~$u_1\longmapsto t$ and~$u_2\longmapsto 1+t$.
The set of places~$\mathcal{S}(\mathfrak{p})$ comprises the usual infinite place of~$\mathbb{F}_2(t)$ corresponding to the degree of a polynomial, together with the places induced by the prime ideals~$(t)$ and~$(1+t)$
in~$\mathbb{F}_2[t^{\pm 1},(1+t)^{\pm 1}]$. These give rise to the Lyapunov list
\begin{eqnarray*}
\mathcal{L}(\alpha)
& = &
\langle
(\log|t|_v, \log|t+1|_v)\mid v=\infty,t,t+1
\rangle\\
& = &
\langle
(\log 2,\log 2),(-\log 2,0), (0,-\log 2)
\rangle.
\end{eqnarray*}
This time, the~$x$- and~$y$-axes
together with the line~$x+y=0$ divide the plane into 6 regions, which we label as before~$\mathcal{C}_1,\dots,\mathcal{C}_6$. In this case, for any~$(x,y)\in\mathbb{R}^2$,
\[
\ent((x,y))=
\left\{
\begin{array}{ll}
(|x| + |y|)\log 2 & \mbox{ for } (x,y)\in\mathcal{C}_1\cup\mathcal{C}_4;\\
|y|\log 2 & \mbox{ for } (x,y)\in\mathcal{C}_2\cup\mathcal{C}_5;\\
|x|\log 2 & \mbox{ for } (x,y)\in\mathcal{C}_3\cup\mathcal{C}_6.\\
\end{array}
\right.
\]
This results in the directional entropy unit ball~$U(\alpha)$ in Figure~\ref{ledrappier_diagram}, and we find
\[
\vol(U(\alpha))=\frac{3}{(\log 2)^2}
\]
so
\[
h^*(\alpha)=\frac{6}{(\log 2)^2}.
\]

\begin{center}
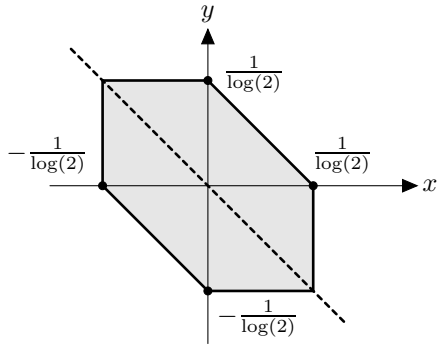
\begin{figure}[h]
\begin{tikzpicture}[line cap=round,line join=round,>=triangle 45,x=1.4cm,y=1.4cm]
\draw[->,color=black] (-1.5,0.0) -- (2.0,0.0);
\draw[->,color=black] (0.0,-1.5) -- (0.0,1.5);
\fill[line width=1.6pt,fill=black,fill opacity=0.1] (-1.0,0.0) -- (-1.0,1.0) -- (0.0,1.0) -- (1.0,0.0) -- (1.0,-1.0) -- (0.0,-1.0) -- cycle;
\draw [line width=1.0pt] (-1.0,0.0)-- (-1.0,1.0);
\draw [line width=1.0pt] (-1.0,1.0)-- (0.0,1.0);
\draw [line width=1.0pt] (0.0,1.0)-- (1.0,0.0);
\draw [line width=1.0pt] (1.0,0.0)-- (1.0,-1.0);
\draw [line width=1.0pt] (1.0,-1.0)-- (0.0,-1.0);
\draw [line width=1.0pt] (0.0,-1.0)-- (-1.0,0.0);
\draw (0.9,0.55) node[anchor=north west] {$\frac{1}{\log(2)}$};
\draw (-2.0,0.55) node[anchor=north west] {$-\frac{1}{\log(2)}$};
\draw (0.06,1.32) node[anchor=north west] {$\frac{1}{\log(2)}$};
\draw (0.0,-1.0) node[anchor=north west] {$-\frac{1}{\log(2)}$};
\draw [line width=1.0pt,dash pattern=on 2pt off 2pt] (-1.3,1.3)-- (1.3,-1.3);
\draw (1.95,0) node[anchor=west] {$x$};
\draw (0,1.47) node[anchor=south] {$y$};
\begin{scriptsize}
\draw [fill=black] (-1.0,0.0) circle (1.5pt);
\draw [fill=black] (0.0,1.0) circle (1.5pt);
\draw [fill=black] (1.0,0.0) circle (1.5pt);
\draw [fill=black] (0.0,-1.0) circle (1.5pt);
\end{scriptsize}
\end{tikzpicture}
\caption{The directional entropy unit ball for Ledrappier's example.\label{ledrappier_diagram}}
\end{figure}
\end{center}
\end{example}

\begin{example}[Commuting toral automorphisms]
The algebraic integers
\[
\xi_1=1+\sqrt{2},
\]
\[
\xi_2=2+\sqrt{5},
\]
and
\[
\xi_3=3+\sqrt{2}\sqrt{5}
\]
are units in the ring~$\mathbb{Z}[\sqrt{2},\sqrt{5}]$
(which is isomorphic to~$\mathbb{Z}^4$
as an additive abelian group).
The automorphisms dual to multiplication by each of~$\xi_1$,~$\xi_2$  and~$\xi_3$ are therefore
commuting hyperbolic automorphisms of the torus~$\mathbb{T}^4$. The Lyapunov list for this example is given by the prime ideal
\[
\mathfrak{p}
=
(u_1^2-2u_1-1,u_2^2-4u_2-1,u_3^2-6u_3-1)
\subset
R_3,
\]
and the set of places~$\mathcal{S}(\mathfrak{p})$,
which in this case correspond to the archimedean absolute values defined by~$|\cdot|_\tau=|\tau(\cdot)|$, where~$\tau\in\mathrm{Gal}(\mathbb{Q}(\sqrt{2},\sqrt{5})|\mathbb{Q})$ and~$|\cdot|$ is the usual absolute value on~$\mathbb{R}$. Therefore,
\[
\mathcal{L}(\alpha)
=
\langle
(\log|\tau(\xi_1)|, \log|\tau(\xi_2)|, \log|\tau(\xi_3)|)\mid
\tau\in\mathrm{Gal}(\mathbb{Q}(\sqrt{2},\sqrt{5})|\mathbb{Q})
\rangle
\]
and the directional entropy function is given by
\[
\ent(\mathbf{t})
=
\sum_{\tau\in\mathrm{Gal}(\mathbb{Q}(\sqrt{2},\sqrt{5})|\mathbb{Q})}
\max\{
(\log|\tau(\xi_1)|, \log|\tau(\xi_2)|, \log|\tau(\xi_3)|)
\cdot
\mathbf{t},
0\}.
\]
This results in the directional entropy unit ball~$U(\alpha)$ shown in Figure~\ref{toral_diagram}, which is a 14-faced polyhedron (a cuboid with 8 tetrahedal corners removed). For this polyhedron,
\[
\vol(U(\alpha))=\frac{5}{6(\log \xi_1)(\log \xi_3)(\log \xi_3)}
\]
and so
\[
h^*(\alpha)=\frac{10}{9(\log \xi_1)(\log \xi_3)(\log \xi_3)}.
\]

\begin{center}
\begin{figure}[h]
\begin{tikzpicture}[line cap=round,line join=round,>=triangle 45,x=1.0cm,y=1.0cm]
\fill[line width=0.4pt,dotted,color=cqcqcq,fill=cqcqcq,fill opacity=1.0] (0.54,3.96) -- (-0.44,2.6) -- (2.54,2.96) -- (2.9,4.24) -- cycle;
\fill[color=aqaqaq,fill=aqaqaq,fill opacity=1.0] (2.9,4.24) -- (3.82,2.94) -- (2.54,2.96) -- cycle;
\fill[color=wqwqwq,fill=wqwqwq,fill opacity=1.0] (0.54,3.96) -- (-1.14,2.24) -- (-0.44,2.6) -- cycle;
\fill[color=yqyqyq,fill=yqyqyq,fill opacity=1.0] (-1.14,2.24) -- (-0.24,0.72) -- (0.6,0.78) -- (-0.44,2.6) -- cycle;
\fill[color=yqyqyq,fill=yqyqyq,fill opacity=1.0] (0.6,0.78) -- (2.4,1.14) -- (3.82,2.94) -- (2.54,2.96) -- cycle;
\fill[color=eqeqeq,fill=eqeqeq,fill opacity=1.0] (-0.44,2.6) -- (0.6,0.78) -- (2.54,2.96) -- cycle;
\draw [dotted] (3.32,3.66)-- (-0.32,1.54);
\draw [dotted] (-0.02,3.44)-- (2.44,2.02);
\draw [dotted] (1.44,0.92)-- (1.44,3.46);
\draw[->] (1.44,3.46)-- (1.44,4.6);
\draw[->] (-0.32,1.54)-- (-1.46,0.88);
\draw[->] (2.44,2.02)-- (4.18,1.0);
\draw (-0.02,3.44)-- (-1.46,4.24);
\draw (3.32,3.66)-- (4.2,4.16);
\draw (1.44,0.92)-- (1.46,0.24);
\draw (-2.8,2.1) node[anchor=north west] {$\frac{1}{2\log(1+\sqrt{2})}$};
\draw (3.1,2.4) node[anchor=north west] {$\frac{1}{2\log(2+\sqrt{5})}$};
\draw (1.7,5.0) node[anchor=north west] {$\frac{1}{2\log(3+\sqrt{2}\sqrt{5})}$};
\draw (-1.4,1.1) node[anchor=north east] {$x$};
\draw (4.1,1.2) node[anchor=north west] {$y$};
\draw (1.45,4.55) node[anchor=south] {$z$};
\begin{scriptsize}
\draw [fill=black] (-0.32,1.54) circle (1.0pt);
\draw [fill=black] (2.44,2.02) circle (1.0pt);
\draw [fill=black] (1.44,3.46) circle (1.0pt);
\end{scriptsize}
\end{tikzpicture}
\caption{The directional entropy unit ball for 3 commuting toral automorphisms.\label{toral_diagram}}
\end{figure}
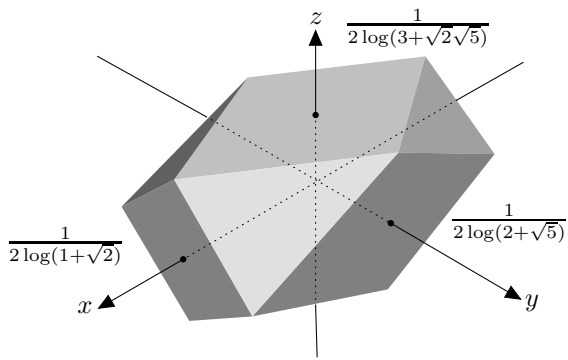\end{center}
\end{example}

\section{Directional entropy and periodic points}~\label{bounds_and_uniformities}

If the directional entropy function~$\ent$ restricted to the unit sphere~$\mathbb{S}^{d-1}$ has a positive lower bound, this may be interpreted as an indication of uniform growth in orbit complexity.

Before stating and proving a non-trivial result in this
direction, we recall that
the entropy of any mixing
automorphism of a compact group lies in the closure of the
set
\[
\{\logmahler(f)\mid\logmahler(f)>0,f\in\mathbb{Z}[x]\},
\]
and the classical \emph{Lehmer problem} asks
if
\[
\inf\{\logmahler(f)\mid\logmahler(f)>0,f\in\mathbb{Z}[x]\}>0.
\]
Moreover, if the topological dimension of~$X$
is finite, then the entropy of
any automorphism of~$X$ is given by an expression
of the form~$\logmahler(f)$ with~$f\in\mathbb{Z}[x]$
of degree no more than~$\dim(X)$.
We refer to the monograph of Everest
and the second author~\cite{MR1700272} for the
background and references on the
properties of
Mahler measure we will need, and to
the survey by the authors and Staines~\cite{MSW} for references
on how this relates to the
dynamical properties of group automorphisms.
What is known is that there is a lower bound
for the logarithmic Mahler measure in the following
sense. Blanskby and Montgomery~\cite{MR0296021}
and Dobrowolski~\cite{MR543210}
give explicit functions~$C:\mathbb{N}\to\mathbb{R}_{>0}$
for which
\[
\inf\{\logmahler(f)\mid\logmahler(f)>0,\deg(f)\le k\}
>C(k)>0.
\]
This means that there is a positive function~$H$ defined
on the space of all compact groups
that admit mixing automorphisms
with the property that any mixing automorphism~$\alpha:X\to X$
has
\[
\ent(\alpha)\ge H(X).
\]

The following result appears in part in the proof
of Theorem~1.1 in the authors' paper~\cite{MR2350424},
and shows
how this relates to the property of mixing in our setting.

\begin{theorem}\label{entropy_bounds_theorem}
Let~$\alpha$ be an entropy rank one~$\mathbb{Z}^d$-action by automorphisms of a compact abelian group~$X$. If~$(X,\alpha)$ satisfies the descending chain condition and is mixing, then there exist positive constants~$C_1, C_2$ such that~$C_1\leqslant\ent(\mathbf{t})\leqslant C_2$ for all~$\mathbf{t}\in\mathbb{S}^{d-1}$.
\end{theorem}

\begin{proof}
Firstly note that since the restriction of~$\ent$ to~$\mathbb{S}^{d-1}$ is a continuous function on a compact set, it is bounded and attains its bounds. Hence, we need only show that~$C_1>0$ if~$\alpha$ is mixing.

Suppose there exists~$\mathbf{t}\in\mathbb{S}^d$ with~$\ent(\mathbf{t})=0$. In the simplest case, the line through~$\mathbf{t}$ intersects a non-zero point~$\mathbf{n}\in\mathbb{Z}^d$ giving~$\ent(\mathbf{n})=0$. It immediately follows that~$\alpha^\mathbf{n}$ is not ergodic and~$\alpha$ is not mixing. If the line does not intersect any such lattice point, then note that for any~$\mathbf{n}\in\mathbf{Z}^d$, using~\eqref{directional_entropy_formula} we have
\begin{equation}\label{and_i_look_across_the_water}
|\ent(\mathbf{n})-\ent(\mathbf{t})|\leqslant
\max_{(\ell_1,\dots,\ell_d)\in\mathcal{L}(\alpha)}
\max_{1\leqslant i\leqslant d}\{|\ell_i|\}\Vert\mathbf{n}-\mathbf{t}\Vert.
\end{equation}
Furthermore, since~$\ent(\lambda\mathbf{t})=0$ for all~$\lambda>0$, we may select a sequence of vectors~$\mathbf{n}_j$ in~$\mathbb{Z}^d$ and a sequence of scalars~$\lambda_j$ with~$\lim_{j\rightarrow\infty}\Vert\mathbf{n}_j-\lambda_j\mathbf{t}\Vert=0$, thus giving
\begin{equation}\label{and_i_think_of_all_the_things}
\lim_{j\rightarrow\infty}\ent(\mathbf{n}_j)=0,
\end{equation}
by~(\ref{and_i_look_across_the_water}).

If~$\ass(M)$ contains maximal ideal~$\mathfrak{p}$, then as noted in Section~\ref{lyupunov_lists},~$\alpha$ cannot be mixing. Hence, for a non-trivial case assume that~$\ass(M)$ contains no maximal ideals, so for each~$\mathfrak{p}\in\ass(M)$,~$R_d/\mathfrak{p}$ has Krull dimension 1. If~$\mathfrak{p}\cap\mathbb{Z}=p\mathbb{Z}$ for some rational prime~$p$, then the entropy addition formula shows that for each non-zero~$\mathbf{n}\in\mathbb{Z}^d$, either~$\ent(\mathbf{n})=0$ or~$\ent(\mathbf{n})\geqslant\log p$. It follows from~(\ref{and_i_think_of_all_the_things}) that~$\ent(\mathbf{n}_j)=0$ for some~$j$, and so~$\alpha^{\mathbf{n}_j}$ is not ergodic and~$\alpha$ is not mixing. The only remaining case is that~$\ass(M)$ contains a prime~$\mathfrak{p}$ for which~$\mathbb{K}(\mathfrak{p})$ is an algebraic number field. In this case, the addition formula for entropy shows that for each non-zero~$\mathbf{n}\in\mathbb{Z}^d$, either~$\ent(\mathbf{n})=0$ or~$\ent(\mathbf{n})\geqslant\logmahler(f)>0$, where~$\logmahler(f)$ denotes the logarithmic Mahler measure of some fixed polynomial~$f$, depending only on~$\mathfrak{p}$, of degree no greater than the degree of the field extension~$\mathbb{K}(\mathfrak{p})|\mathbb{Q}$. Reasoning identical to the previous case can now be applied. As noted at the start of this section,
there are explicit lower bounds for the non-zero logarithmic Mahler measure of polynomials of bounded degree.
\end{proof}

\begin{remark}
If~$\alpha$ is an entropy rank one~$\mathbb{Z}^d$-action by automorphisms of a compact abelian group~$X$, and~$\alpha$ is both non-mixing and \emph{irreducible} (that is, the only closed~$\alpha$-invariant subgroups of X are finite), then there exists a non-zero~$\mathbf{n}\in\mathbb{Z}^d$ such that~$\alpha^\mathbf{n}$ is not ergodic, so~$\ent(\mathbf{n})=0$, by irreducibility.  Since~\eqref{directional_entropy_formula} gives
\begin{equation}\label{scaling_formula}
\ent(\mathbf{n})=\Vert\mathbf{n}\Vert\ent(\widehat{\mathbf{n}}),
\end{equation}
it follows that~$\ent(\widehat{\mathbf{n}})=0$, where~$\widehat{\mathbf{n}}$ is a unit vector in the direction of~$\mathbf{n}$. Thus, the directional entropy function cannot be bounded away from zero in this case.
\end{remark}

For any closed~$\alpha$-invariant subgroup~$Y\subset X$, let~$\alpha_Y$ and~$\alpha_{X/Y}$ denote the induced actions of~$\alpha$ on~$Y$ and~$X/Y$ respectively. A standard dimension argument shows that for any mixing~$\mathbb{Z}^d$-action~$\alpha$ by automorphisms of a finite-dimensional compact connected abelian group~$X$, there exists a closed~$\alpha$-invariant subgroup~$Y\subset X$ such that~$(X/Y,\alpha_{X/Y})$ is mixing, satisfies the descending chain condition and, for all~$\mathbf{n}\in\mathbb{Z}^d$,~$h(\alpha_X^\mathbf{n})=
h(\alpha_{X/Y}^\mathbf{n})$, where~$h$ denotes the usual topological entropy. This follows from the Yuzvinski{\u\i} entropy addition formula and the fact that the subgroup~$Y$ may be selected so that~$h(\alpha_{Y}^\mathbf{n})=0$ for all~$\mathbf{n}\in\mathbb{Z}^d$.  Subsequently, the directional entropy functions are identical for~$\alpha_X$ and~$\alpha_{X/Y}$. Hence, we also have the following.

\begin{corollary}\label{solenoidal_bound_corollary}
If~$\alpha$ is a mixing~$\mathbb{Z}^d$-action by automorphisms of a compact connected finite-dimensional abelian group~$X$, then there exists a constant~$C>0$ such that~$\ent(\mathbf{t})>C$
for all~$\mathbf{t}\in\mathbb{S}^{d-1}$.
\end{corollary}

\begin{corollary}
Let~$\alpha$ be a mixing~$\mathbb{Z}^d$-action by automorphisms of a compact abelian group~$X$. If~$X$ is connected and finite-dimensional, or if~$(X,\alpha)$ has entropy rank one and satisfies the descending chain condition, then the directional entropy function for~$\alpha$ defines a norm on~$\mathbb{R}^d$.
\end{corollary}

\begin{proof}
The formula~(\ref{directional_entropy_formula}) shows that~$\ent$ defines a semi-norm on~$\mathbb{R}^d$ and Theorem~\ref{entropy_bounds_theorem} and Corollary~\ref{solenoidal_bound_corollary}  show that~$\ent(\mathbf{t})=0$ if and only if~$\mathbf{t}=0$.
\end{proof}

The following theorem, taken from
work of the authors~\cite{MR2350424}, shows that the uniform growth just described is also evident in the periodic point data for elements of the action.
Interestingly, Pollicott~\cite{pollicott_growth_of_periodic_points} later used a different method to achieve a similar result for commuting toral automorphisms.
In the next section, we will obtain a more precise result, relating the uniform growth in the number of periodic points directly to the Fried average entropy.
For a $\mathbb{Z}^d$-action $\alpha$ by automorphisms of a compact abelian group $X$ and for any $\mathbf{n}\in\mathbf{Z}^d$, let
\[
\fix(\alpha^{\mathbf{n}})=\{x\in X\mid\alpha^{\mathbf{n}}(x)=x\}.
\]
The fixed points for individual elements of the action are called \emph{periodic points}. Although there is another notion of periodic point for a lattice in a~$\mathbb{Z}^d$-action
used in work of Lind, Schmidt and the second
author~\cite{MR1062797},
the periodic points for lattices are more sparse when~$\alpha$ is an entropy rank one action
and they do not reveal the rates of growth that are our primary
interest here. For a more precise statement and for further details, see the work of the first author~\cite{miles_lind_zeta}.
There is a convenient formula
for~$|\fix(\alpha^{\mathbf{n}})|$ using a product over places of global
fields due to the first author~\cite[Sec.~3]{MR2308145}, originating in earlier work of Chothi,
Everest and the second author~\cite{MR1461206}, that is an essential starting point in obtaining results such as the following. We will not use this formula explicitly until Section~\ref{fried_entropy_and_fixed_points}, and reserve its statement to the beginning of the proof of Theorem~\ref{second_fixed_point_uniformity_theorem}.

For a sequence~$(\mathbf{n}_j)$ in~$\mathbb{Z}^d$, write~$\mathbf{n}_j\rightarrow\infty$ to mean that~$\Vert\mathbf{n}_j\Vert\rightarrow\infty$
as~$j\rightarrow\infty$. This can also
be thought of as the sequence leaving finite sets.

\begin{theorem}[Miles and Ward~\cite{MR2350424}]\label{first_fixed_point_uniformity_theorem}
Let~$\alpha$ be a mixing~$\mathbb{Z}^d$-action by automorphisms of a compact abelian group~$X$. If~$(X,\alpha)$ has entropy rank one and satisfies the descending chain condition, then the quantities
\[
\limsup_{\mathbf{n}\rightarrow\infty}
\frac{1}{\Vert\mathbf{n}\Vert}\log |\fix(\alpha^\mathbf{n})|
\]
and
\[
\liminf_{\mathbf{n}\rightarrow\infty}
\frac{1}{\Vert\mathbf{n}\Vert}\log |\fix(\alpha^\mathbf{n})|
\]
are finite and strictly positive.
\end{theorem}

The uniform growth in the number of periodic points demonstrated by Theorem~\ref{first_fixed_point_uniformity_theorem}
may not be immediately evident in a given
example. The array in Figure~\ref{23array}
is taken from~\cite{MR2350424} and shows~$|\fix(\alpha^{(n_1,n_2)})|$ for
the~$\times2,\times3$-action in Example~\ref{times2_times3_example}, with~$-5\leqslant n_1\leqslant 5$ and~$0\leqslant n_2\leqslant 5$. Note that for lattice points close to the line~$n_1\log 2 + n_2\log 3 = 0$ (indicated in
italics), it is not obvious that we have an exponential growth in the number of periodic points. The fact that we do have this exponential growth independently
of the direction chosen (or indeed, the path chosen) is a Diophantine
property of Baker type: it is in essence a statement about the possible
size of expressions like~$|2^a3^b-1|$ for integer points~$(a,b)$ chosen
to lie close to the line~$2^x3^y=1$ in terms of~$\Vert(a,b)\Vert$.

\begin{center}
\begin{figure}
\begin{tabular}{ccccccccccc}
211 & 227 & 235 & 239 & 241 & 121 & 485 & 971 & 1943 & 3887 & 7775\\
49 & 65 & 73 & 77 & 79 & 5 & 161 & 323 & 647 & 1295 & 2591\\
{\it 5} & 11 & 19 & 23 & 25 & 13 & 53 & 107 & 215 & 431 & 863\\
23 & {\it 7} & {\it 1} & 5 & 7 & 1 & 17 & 35 & 71 & 143 & 287\\
29 & 13 & 5 & {\it 1} & {\it 1} & 1 & 5 & 11 & 23 & 47 & 95\\
31 & 5 & 7 & 1 & 1 &~$\infty$ & 1 & 1 & 7 & 5 & 31
\end{tabular}
\caption{\label{23array}Periodic point counts in the $\times2,\times3$-action.}
\end{figure}
\end{center}

The periodic point data for elements of an entropy rank one action, such as that shown in Figure~\ref{23array}, featured in another role in the
earlier work of the authors~\cite{MR2279271}, where it was used to detect expansive subdynamics. Using results from work of the first author~\cite{MR2308145},
the poles and zeros of associated dynamical zeta functions
were used to reconstruct non-expansive directions and
expansive components (analogous to Weyl chambers for a smooth~$\mathbb{Z}^d\times\mathbb{R}^n$-action on a compact manifold).
We note the connection here with the directional entropy function, restricted to~$\mathbb{S}^{d-1}$, which is hiding behind the scenes. Figure~\ref{pole_and_zero_picture} is taken from~\cite{MR2279271} and gives a brief illustration of the approach used there via Example~\ref{times2_times3_example}.
For each region~$\mathcal{C}_i$ described in the example, let~$\mathcal{C}'_i$ denote the set of integral points in the open cone formed by the removal of the axes and the irrational line shown in Figure~\ref{times2_times3_diagram}.
If~$(n_1,n_2)\in\mathcal{C}_i$,~$1\leqslant i\leqslant 6$, then the automorphism~$\alpha^{(n_1,n_2)}$ is expansive and has a rational zeta function given by
\[
\zeta_{(n_1,n_2)}(z)
=
\left\{
\begin{array}{ll}
(1-z)/(1-2^{|n_1|}3^{|n_2|}z) & \mbox{ if } (n_1,n_2)\in \mathcal{C}'_1\cup \mathcal{C}'_4;\\
(1-2^{|n_1|}z)/(1-3^{|n_2|}z) & \mbox{ if }(n_1,n_2)\in \mathcal{C}'_2\cup \mathcal{C}'_5;\\
(1-3^{|n_2|}z)/(1-2^{|n_1|}z) & \mbox{ if }(n_1,n_2)\in \mathcal{C}'_3\cup \mathcal{C}'_6.
\end{array}
\right.
\]
For each~$\mathbf{n}\in\mathbb{Z}^2$, let~$\Psi_\mathbf{n}$ denote the set of poles and zeros of~$\zeta_\mathbf{n}$. Illustrated in Figure~\ref{pole_and_zero_picture} is the closure of the set
\[
\Omega_{\times 2 \times 3}=
\{(\widehat{\mathbf{n}}, |z|^{1/\Vert\mathbf{n}\Vert})\mid
 \mathbf{n}\in\mathcal{C}'_1\cup\dots\cup\mathcal{C}'_6, z\in\Psi_\mathbf{n}\}
\]
where we have parameterized~$\mathbb{S}^1$ by the angles~$0\leqslant\theta<2\pi$ along the horizontal axis and shown the second coordinate of~$\Omega_{\times 2 \times 3}$ on the vertical axis. Note the appearance of the directional entropy function: the lower curve in Figure~\ref{pole_and_zero_picture} is given by~$\ent((\sin \theta,\cos\theta))=-\log y$.

\begin{figure}[h]
\begin{center}
\begin{tikzpicture}[line cap=round,line join=round,>=triangle 45,x=1.0cm,y=4.0cm]
\draw[->,color=black] (-0.2,0.0) -- (6.4,0.0);
\foreach \x in {1.0,2.0,3.0,4.0,5.0,6.0}
\draw[shift={(\x,0)},color=black] (0pt,2pt) -- (0pt,-2pt) node[below] {\footnotesize~$\x$};
\draw[->,color=black] (0.0,-0.1) -- (0.0,1.1);
\foreach \y in {0.2,0.4,0.6,0.8, 1.0}
\draw[shift={(0,\y)},color=black] (2pt,0pt) -- (-2pt,0pt) node[left] {\footnotesize~$\y$};
\draw[color=black] (0pt,-10pt) node[right] {\footnotesize~$0$};
\clip(-0.1,-0.1) rectangle (6.8,1.2);
\draw (6.35,0.0) node[anchor=west] {$\theta$};
\draw (0.0,1.09) node[anchor=south] {$y$};
\draw[smooth,samples=200,domain=0:1.571] plot(\x,{1.0/(2.0^(abs(cos(((\x))*180/pi)))*3.0^(abs(sin(((\x))*180/pi))))});
\draw[smooth,samples=200,domain=3.142:4.712] plot(\x,{1.0/(2.0^(abs(cos(((\x))*180/pi)))*3.0^(abs(sin(((\x))*180/pi))))});
\draw[smooth,samples=200,domain=0:1.571] plot(\x,{1.0});
\draw[smooth,samples=200,domain=3.142:4.712] plot(\x,{1.0});
\draw[smooth,samples=200,domain=1.571:3.142] plot(\x,{1.0/(2.0^(abs(cos(((\x))*180/pi))))});
\draw[smooth,samples=200,domain=1.571:3.142] plot(\x,{1.0/(3.0^(abs(sin(((\x))*180/pi))))});
\draw[smooth,samples=200,domain=4.712:6.283] plot(\x,{1.0/(2.0^(abs(cos(((\x))*180/pi))))});
\draw[smooth,samples=200,domain=4.712:6.283] plot(\x,{1.0/(3.0^(abs(sin(((\x))*180/pi))))});
\end{tikzpicture}
\caption{The closure of~$\Omega_{\times 2 \times 3}$ summarizing directional pole and zero data.\label{pole_and_zero_picture}}
\end{center}
\end{figure}
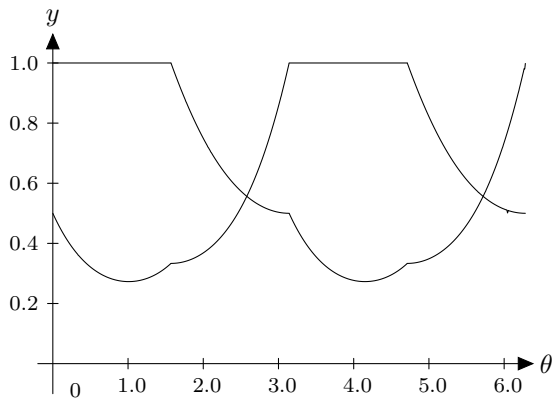

\section{Uniform growth of periodic points and Fried average entropy}~\label{fried_entropy_and_fixed_points}

As before, consider a mixing~$\mathbb{Z}^d$-action~$\alpha$ by automorphisms of a compact abelian group~$X$. Assume that~$X$ is connected and~$(X,\alpha)$ satisfies the descending chain condition.
We will be interested in the volume of the convex hull of the set of points in~$\mathbb{Z}^d$ that fix at most~$N>0$ points of~$X$ under the action. Hence, define~$\chull(N)$ to be the convex hull of the set
\[
\{\mathbf{n}\in\mathbb{Z}^d:|\fix(\alpha^\mathbf{n})|\leqslant N\}.
\]
The following theorem shows how the asymptotic volume of this convex hull is directly related to the Fried average entropy.

\begin{theorem}\label{second_fixed_point_uniformity_theorem}
Let~$\alpha$ be a mixing~$\mathbb{Z}^d$-action by automorphisms of a compact connected abelian group~$X$ of finite topological dimension such that~$(X,\alpha)$ satisfies the descending chain condition. Then
\[
\lim_{N\rightarrow\infty}\frac{\vol(\chull(N))}{(\log N)^d}=\vol(U(\alpha)).
\]
Hence, dividing the volume of the generalized octahedron in~$\mathbb{R}^d$ by this limit gives the Fried average entropy.
\end{theorem}

Before proving the theorem, we return to the familiar Example~\ref{times2_times3_example}.

\begin{example}
Since~$|\fix(\alpha^{-\mathbf{n}})|=|\fix(\alpha^{\mathbf{n}})|$, for all~$\mathbf{n}\in\mathbb{Z}^d$, we need only consider the
periodic point counts given in Table~\ref{table_of_fixed_point_counts}, where~$\mathcal{C}_1$,~$\mathcal{C}_2$ and~$\mathcal{C}_3$ refer to the regions in the upper half plane described in
Example~\ref{times2_times3_example} and shown in  Figure~\ref{times2_times3_diagram}.
\begin{center}
\begin{table}[h]
\caption{Periodic point counts for the~$\times 2$~$\times 3$ example.\label{table_of_fixed_point_counts}}
\begin{tabular}{l|l}
Location of~$(x,y)\in\mathbb{Z}^2$ &~$|\fix(\alpha^{(x,y)})|$\\
\hline
Interior of~$\mathcal{C}_1$ &~$2^x3^y-1$\\
Interior of~$\mathcal{C}_2$ &~$3^y|1-2^x 3^{-y}|$\\
Interior of~$\mathcal{C}_3$ &~$2^{|x|}|1-2^x 3^y|$\\
$x$-axis &~$(2^{|x|}-1)|2^{|x|}-1|_3$\\
$y$-axis &~$(3^{y}-1)|3^{y}-1|_2$\\
\end{tabular}
\end{table}
\end{center}
Now consider, for example, the integral points~$(x,y)$ in the interior of~$\mathcal{C}_3$ that satisfy~$|\fix(\alpha^{(x,y)})|\leqslant N$. Firstly, note that~$|1-2^x 3^y|<1$ for all points~$(x,y)$ in the interior of~$\mathcal{C}_3$. On the other hand, by Baker's theorem (see for example
Baker's monograph~\cite{MR1074572}
or the monograph of Everest
and the second author~\cite{MR1700272}),
there exist positive constants~$A$ and~$B$ such that when~$|x|$  is large,
\[
|1-2^x 3^y|>A/|x|^B.
\]
Suppose also that~$|x|>\frac{1}{\log 2}( \log N +(\log N)^{\delta})$, where~$0<\delta<1$ is close to~1.
Then, when~$N$ is large, in the interior of~$\mathcal{C}_3$
\[
|\fix(\alpha^{(x,y)})|
=
2^{|x|}|1-2^x 3^y|
>
N\cdot
\frac{A(\log 2)^B\cdot e^{(\log N)^{\delta}}}
{(\log N +(\log N)^{\delta})^B}
>
N
\]
and along the negative~$x$-axis
\[
|\fix(\alpha^{(x,y)})|=
(2^{|x|}-1)|2^{|x|}-1|_3
>
(N\cdot e^{(\log N)^{\delta}}-1) \frac{\log 2}{3(\log N +(\log N)^{\delta})}
>
N.
\]

Similar calculations can be performed in the other five regions described in Example~\ref{times2_times3_example} and it follows that~$\chull(N)$ lies inside a convex region similar to the one shaded in Figure~\ref{times2_times3_diagram}, with a boundary that intercepts the axes at~$x=\pm\frac{1}{\log 2}( \log N +(\log N)^{\delta})$ and ~$y=\pm\frac{1}{\log 3}( \log N +(\log N)^{\delta})$. Also, if~$|x|<\frac{\log N}{\log 2}$, then
$|\fix(\alpha^{(x,y)})|<N$ (and this is also true if~$(x,y)$ lies on the negative~$x$-axis with~$|x|<\frac{\log N}{\log 2}$). Similar reasoning applies for the other five regions, so~$\chull(N)$ contains a convex region similar to the one shaded in Figure~\ref{times2_times3_diagram} with a boundary that intercepts the axes at~$x=\pm \frac{\log N}{\log 2}$ and~$y=\pm\frac{\log N}{\log 3}$. Thus,
\[
\frac{3(\log N)^2}{(\log 2)(\log 3)}<\vol(\chull(N))<\frac{3(\log N +(\log N)^{\delta})^2}{(\log 2)(\log 3)},
\]
giving
\[
\lim_{N\rightarrow\infty}\frac{\vol(\chull(N))}{(\log N)^2}=
\frac{3}{(\log 2)(\log 3)}
=\vol(U(\alpha)).
\]
\end{example}

\begin{proof}[Proof of Theorem~\ref{second_fixed_point_uniformity_theorem}]
We use the notation introduced for entropy rank one actions in Section~\ref{lyupunov_lists}. For any non-zero~$\mathbf{n}\in\mathbb{Z}^d$, the periodic point counting formul\ae{\ }from
Miles~\cite[Sec.~3]{MR2308145} give
\begin{equation}\label{what_youre_doing_in_my_head}
|\fix(\alpha^{\mathbf{n}})|=
\prod_{\mathfrak{p}\in\ass(M)}
\prod_{v\in \mathcal{S}(\mathfrak{p})}
|\overline{u}^\mathbf{n}-1|^{m(\mathfrak{p})}_v
\end{equation}
where~$\overline{u}^\mathbf{n}$ denotes the image of~$u_1^{n_1}\cdots u_d^{n_d}$ in the domain~$R_d/\mathfrak{p}$
and~$m(\mathfrak{p})$ is the dimension of the~$\mathbb{K}(\mathfrak{p})$-vector space~$M\otimes\mathbb{K}(\mathfrak{p})$. Using~(\ref{directional_entropy_formula}),~(\ref{what_youre_doing_in_my_head}) can be rewritten in the form
\begin{equation}\label{entropy_fixed_point_formula}
|\fix(\alpha^{\mathbf{n}})|=e^{\ent(\mathbf{n})}g(\mathbf{n}),
\end{equation}
where~$g:\mathbb{Z}^d\rightarrow \mathbb{R}_{\geqslant 0}$ is given by
\[
g(\mathbf{n})
=
\prod_{\mathfrak{p}\in\ass(M)}
\prod_{v\in \mathcal{S}(\mathfrak{p})}
|1-\phi_{\mathfrak{p},v}(\mathbf{n})|^{m(\mathfrak{p})}_v
\]
and
\[
\phi_{\mathfrak{p},v}(\mathbf{n})=\left\{
\begin{array}{ll}
\overline{u}^{-\mathbf{n}} & \mbox{ if } |\overline{u}^{\mathbf{n}}|_v>1,\\
\overline{u}^{\mathbf{n}} & \mbox{ if }|\overline{u}^{\mathbf{n}}|_v\leqslant 1.
\end{array}
\right.
\]
Note that~$g(\mathbf{n})<2^D$, where~$D=|\ass(M)|\prod_{\mathfrak{p}\in\ass(M)}m(\mathfrak{p})$, and
that the mixing assumption ensures that~$\overline{u}^\mathbf{n}\neq 1$, so~$g(\mathbf{n})>0$. Furthermore, just as in the proof
of Theorem~1.1 in the work of the
authors~\cite{MR2350424},
the theorems of Baker~\cite{MR1074572} and
Yu~\cite{MR1055245} can be used to show that there exist positive constants~$A$ and~$B$ such that
\[
g(\mathbf{n})> A/\max\{|n_i|\}^B.
\]
Denote the closed ball of radius~$r>0$ in the directional entropy norm by
\[
B(r)=\{\mathbf{t}\in\mathbb{R}^d\mid\ent(\mathbf{t})\leqslant r\},
\]
so~$U(\alpha)=B(1)$. Each such ball is a convex polytope in~$\mathbb{R}^d$. We will show that for large~$N$,~$\chull(N)$ is trapped between~$B(\log(N/2^D))$ and~$B(\log N +(\log N)^{\delta})$, where~$0<\delta<1$ is a fixed constant close to~1.

On the one hand, if~$N$ is large enough so that~$N>2^D$ and~$\ent(\mathbf{n})\leqslant\log(N/2^D)$, then~(\ref{entropy_fixed_point_formula}) gives
$|\fix(\alpha^{\mathbf{n}})|\leqslant N$,
so~$B(\log(N/2^D))\subset \chull(N)$.

On the other hand, assume that~$\ent(\mathbf{n})>\log N +(\log N)^{\delta}$. Theorem~\ref{entropy_bounds_theorem} shows that there is a constant~$C_1>0$ such that~$\ent(\mathbf{t})>C_1$ for all~$\mathbf{t}\in\mathbb{S}^{d-1}$. Therefore, for any non-zero~$\mathbf{n}\in\mathbb{Z}^d$,
\[
\ent(\mathbf{n})=
\Vert\mathbf{n}\Vert
\sum_{\boldsymbol{\ell}\in\mathcal{L}(\alpha)}
\max\{\boldsymbol{\ell}\cdot\mathbf{\widehat{n}},0\}
>C_1\Vert\mathbf{n}\Vert\geqslant C_1\max\{|n_i|\}.
\]
Hence, using~(\ref{entropy_fixed_point_formula}),
$|\fix(\alpha^{\mathbf{n}})|$ is at least as large as
\[
e^{\ent(\mathbf{n})} g(\mathbf{n})
>
\frac{Ae^{\ent(\mathbf{n})}}{\max\{|n_i|\}^B}
>
\frac{AC_1^Be^{\ent(\mathbf{n})}}{\ent(\mathbf{n})^B}
>
N\cdot\frac{AC_1^B e^{(\log N)^{\delta}}}{(\log N +(\log N)^{\delta})^B}
>
N,
\]
when~$N$ is large. That is,
\[
\{\mathbf{n}\in\mathbb{Z}^d\mid|\fix(\alpha^\mathbf{n})|\leqslant N\}
\subset
B(\log N +(\log N)^{\delta})
\]
and, since this ball is convex,
\[
\chull(N)\subset B(\log N +(\log N)^{\delta}).
\]
It follows that
\[
\vol(B(\log(N/2^D)))
\leqslant
\vol(\chull(N))
\leqslant
\vol(B(\log N +(\log N)^{\delta})).
\]
Thus,
\[
\vol(B(1))\frac{(\log N-D\log 2)^d}{(\log N)^d}
\leqslant
\frac{\vol(\chull(N))}{(\log N)^d}
\leqslant
\vol(B(1))\frac{(\log N +(\log N)^{\delta})^d}{(\log N)^d},
\]
so that
\[
\lim_{N\rightarrow\infty}\frac{\vol(\chull(N))}{(\log N)^d}=\vol(B(1)).
\]
\end{proof}

The assumption of connectedness is essential in Theorem~\ref{second_fixed_point_uniformity_theorem}, as the following example demonstrates.

\begin{example}\label{fudged_ledrappier}
Consider again Ledrappier's example (Example~\ref{ledrappiers_example}).
Here, the main obstacle to obtaining a result like Theorem~\ref{second_fixed_point_uniformity_theorem}
is that~$\{\mathbf{n}\in\mathbb{Z}^2\mid|\fix(\alpha^\mathbf{n})|\leqslant N\}$ is infinite for all~$N\geqslant 1$. To see this, let~$k$ be any positive integer and note that~(\ref{what_youre_doing_in_my_head}) gives
\[
|\fix(\alpha^{(2^k,0)})|=
|t^{2^k}-1|_\infty
|t^{2^k}-1|_t
|t^{2^k}-1|_{t+1}=2^{2^k}\cdot 1\cdot 2^{-2^k}=1,
\]
since~$t^{2^k}-1=(t+1)^{2^k}$.
This shortage of periodic points arises for certain non-expansive elements of the action~$\alpha^{(x,y)}$ corresponding to integral points~$(x,y)$ on the axes and on the line~$x+y=0$. Such non-expansive automorphisms of zero-dimensional groups typically display erratic patterns of periodic point growth,
as shown in work of Chothi, Everest and the
second author~\cite{MR1461206}.
This complication could be avoided by imposing a restriction to only include those elements of the action that are expansive, that is, by replacing~$\chull(N)$ with the convex hull of
\[
\chull'(N)=\{\mathbf{n}\in\mathbb{Z}^2\mid\alpha^\mathbf{n}\mbox{ is expansive and }|\fix(\alpha^\mathbf{n})|\leqslant N\}.
\]
However, in this setting, the simple formula
\begin{equation}\label{equation:entropyandperiodicpoints}
|\fix(\alpha^\mathbf{n})|=e^{\ent(\mathbf{n})}
\end{equation}
applies to these expansive elements, so that
\[
\lim_{N\rightarrow\infty}\frac{\vol(\chull'(N))}{(\log N)^2}=\vol(U(\alpha))=\frac{3}{(\log 2)^2}
\]
follows immediately, and is not particularly surprising.
\end{example}

The reader may have noticed that
the last identity
obtained in Example~\ref{fudged_ledrappier} is more
closely related to the
geometry of the acting group than the
dynamical properties of the action.
This tension between complexity arising from the
properties of the acting group and dynamical
properties of the action is a constraint on
attempts to generalize results to~$\mathbb{Z}^d$-actions.
For example, the authors extended the dynamical
Mertens' theorem
and other orbit growth results
to nilpotent group actions~\cite{MR2465676}
but only where the action is a full shift, so
that the simple relationship~\eqref{equation:entropyandperiodicpoints}
holds for any lattice in the acting group.
For more complex actions only partial results
could be obtained for~$\mathbb{Z}^2$-actions, and
the authors use examples to show that
several new phenomena arise~\cite{MR2650793}.

\begin{example}
The assumption that the system satisfies the
descending chain condition is also essential in
Theorem~\ref{second_fixed_point_uniformity_theorem}.
To see this, consider the $\mathbb{Z}^2$-action generated
by the maps~$x\longmapsto2x$
and~$x\longmapsto3x$ on the solenoid dual to the
ring
\[
\mathbb{Z}_{(5)}
=
\{x\in\mathbb{Q}\mid\vert x\vert_5\le1\}
=
\{x=\textstyle\frac{a}{b}\in\mathbb{Q}\mid
a\in\mathbb{Z}, \gcd(5,b)=1\}.
\]
The resulting dynamical system has the well known $\times 2$, $\times 3$ example (Ex.~\ref{times2_times3_example}) as an irreducible factor, but has a significantly depleted supply of periodic points. To see this, consider the homomorphism $\chi:\mathbb{Z}^2\rightarrow\mathbb{F}_5^\times$ given by $\chi(x,y)=2^x 3^y\mod 5$. Since~\cite{MR2736895} shows that for any $(x,y)\in\mathbb{Z}^2$, $|\fix(\alpha^{(x,y)})|=|2^x 3^y-1|_5^{-1}$,
we have
\[
|\fix(\alpha^{(x,y)})|\neq1\Longleftrightarrow(x,y)\in\ker(\chi)=\{(j,j+4k):j,k\in\mathbb{Z}\}.
\]
In particular, the only points in~$\mathbb{Z}^2$ with non-trivial periodic points lie in a subgroup of index~$4$. Hence,
if~$(x,y)$ lies outside this subgroup,
then~$|\fix(\alpha^{(x,y)})|=1$, and
so~$\chull(N)$ is undefined for all~$N\in\mathbb{N}$.
\end{example}

\section{Synchronization points}

There is an essential change in perspective between relational entropy, as considered in the introduction, and the Fried average entropy, in the sense that one is an invariant of a generating set of~$d$ commuting automorphisms and the other is an invariant of the~$\mathbb{Z}^d$-action they generate. Although in the previous section it was shown that the Fried average entropy can be obtained naturally using period point counts, there does not appear to be an equally natural way to obtain relational entropy using this data. On the other hand, there is a straightforward way to define non-trivial invariants that capture shared periodic behaviour and reflect related growth rates for a set of generating maps, using a generalization of the notion of a periodic point, and this is the focus of
work of the first author~\cite{MR3074380}.

Let~$\Lambda$ denote any finite set of endomorphisms of a common phase space~$X$ and assume~$|\Lambda|\geq 2$. Two endomorphisms~$\alpha, \beta\in \Lambda$ have a pair of
\emph{synchronous orbits} at time~$n\in\mathbb{N}$ if
there exists~$x\in X$ such that~$\alpha^n(x)=\beta^n(x)$. Hence, define the set of \emph{weak synchronization points} for~$\Lambda$ at time~$n$ by
\[
\wsyn_n(\Lambda)=\{x\in X:\alpha^n(x)=\beta^n(x)\mbox{ for some }\alpha,\beta\in\Lambda, \alpha\neq\beta\},
\]
and the set of \emph{strong synchronization points} for~$\Lambda$ at time~$n$ by
\[
\ssyn_n(\Lambda)=\{x\in X:\alpha^n(x)=\beta^n(x)\mbox{ for all }\alpha,\beta\in\Lambda, \alpha\neq\beta\}.
\]
Clearly, for all~$n\in\mathbb{N}$,~$\ssyn_n(\Lambda)=\wsyn_n(\Lambda)$ when~$|\Lambda|=2$ and~$\ssyn_n(\Lambda)\subset\wsyn_n(\Lambda)$ more generally. If~$\Lambda$ contains the identity map,~$\ssyn_n(\Lambda)$ comprises the points simultaneously fixed by the~$n$th iterates of the non-trivial maps in~$\Lambda$, so
\[
\ssyn_n(\alpha,\id)=\wsyn_n(\alpha,\id)=\fix(\alpha^n).
\]
Consequently, synchronization points generalize periodic points. We are thus led to consider the existence and value of the growth rates
\[
\wrad_\Lambda=\lim_{n\rightarrow\infty}|\wsyn_n(\Lambda)|^{1/n}
\]
and
\[
\srad_\Lambda=\lim_{n\rightarrow\infty}|\ssyn_n(\Lambda)|^{1/n}.
\]
For example, if~$\Lambda=\{\times 2,\times 3\}$ and~$X$ is the solenoid dual to~$\mathbb{Z}[\frac{1}{6}]$ described in Example~\ref{times2_times3_example}, then~$|\wsyn_n(\Lambda)|=|\ssyn_n(\Lambda)|=|\{x\in X:2^nx=3^nx\}|=3^n-2^n$ and
\[
\wrad_\Lambda=\srad_\Lambda=\lim_{n\rightarrow\infty}(3^n-2^n)^{1/n}=3.
\]

For two
automorphisms~$\alpha$ and~$\beta$ of a compact abelian group, the
properties of the
sequence~$(|\wsyn_n(\alpha, \beta)|)$
can often be easily described. For example, when~$\alpha$ and~$\beta$ are two (not necessarily commuting) automorphisms of the 2-torus, Section~3 in~\cite{MR3074380} shows that the sequence~$(|\wsyn_n(\alpha, \beta)|)$ is given by a linear recurrence sequence exhibiting either exponential or polynomial growth (or a mixture of the two), depending on the nature of the defining automorphisms.

If~$\Lambda$ is simply a finite set of
two or more commuting homeomorphisms, then it is also possible to give a precise description of~$\wrad_\Lambda$ in terms of topological entropy, under the hypotheses of the following theorem.

\begin{theorem}[Miles~\cite{MR3074380}]
Let~$\Lambda$ be a finite set of commuting homeomorphisms of a compact metric space~$X$ such that for any two distinct maps~$\alpha, \beta\in\Lambda$, the homeomorphism~$\beta^{-1}\alpha$ is expansive. Then
\begin{equation}\label{synch_points_and_entropy_formula}
\wrad_\Lambda=\max_{\alpha,\beta\in\Lambda,\;\alpha\neq\beta}\{\exp(h(\beta^{-1}\alpha))\},
\end{equation}
where~$h$ denotes the topological entropy.
\end{theorem}

When dealing with automorphisms of a compact abelian group, methods from commutative algebra provide formul\ae{\ }for counting synchronization points similar to those available for counting periodic points, using products over places of global fields. For example, if~$\Lambda$ is a finite set of automorphisms that generates a mixing~$\mathbb{Z}^d$-action on a solenoid,
then
a result of the first author~\cite[Sec.~4]{MR3074380} shows there are associated algebraic number fields~$\mathbb{K}_1,\dots,\mathbb{K}_r$, associated sets of places~$P_i$ of each field~$\mathbb{K}_i$, and distinguished elements~$\lambda_{\alpha, i}\in\mathbb{K}_i$,~$\alpha\in\Lambda$,~$1\leqslant i\leqslant r$, such that for any~$\alpha,\beta\in\Lambda$,~$\alpha\neq\beta$,
\begin{equation}\label{synch_point_counting_formula}
|\wsyn_n(\alpha,\beta)|=
\prod_{i=1}^{r}
\prod_{v\in P_i}
|\lambda_{\alpha, i}^n-\lambda_{\beta, i}^n|_v^{-1}.
\end{equation}
Amongst several useful consequences, this means that
Baker-type estimates become available once more to deal
with the erratic behaviour in~$|\wsyn_n(\alpha,\beta)|$ caused
by factors where~$|\lambda_{\alpha, i}|_v=|\lambda_{\beta, i}|_v$.

A typical assumption in the context of group automorphisms, that turns out to be somewhat weaker than the expansiveness assumption of the theorem above, is that
all but finitely many of the places of the fields~$\mathbb{K}_i$ appear in the counting formula~(\ref{synch_point_counting_formula}). Under this assumption,~\cite[Th.~4.7]{MR3074380} shows that~(\ref{synch_points_and_entropy_formula}) holds, and can be expressed as
\[
\wrad_\Lambda=\max_{\alpha,\beta\in\Lambda,\;\alpha\neq\beta}
\{
H_i(\lambda_{\alpha, i}:\lambda_{\beta, i})
\},
\]
where~$H_i$ denotes the projective height on~$\mathbb{P}^{1}(\mathbb{K}_i)$ (see~\cite[Sec.~4]{MR3074380}).

Turning our attention now to the growth rate of strong synchronization points, a simple non-trivial example is given in~\cite[Sec.~2]{MR3074380} with~$X=\mathbb{Z}[\frac{1}{6}]$ and
\[
\Lambda=\{\times 1,\times 2,\times 3\},
\]
and there it is shown that
\[
|\ssyn_n(\Lambda)|=\gcd(3^n-1,2^n-1).
\]
Such sequences feature in a circle of problems of active interest
in number theory and are less well understood than
linear recurrence sequences. In particular,
while it is well-known that~$\gcd(3^n-1,2^n-1)$ is unbounded,
it is not known whether or
not~$\gcd(3^n-1,2^n-1)=1$ infinitely often. Ailon and Rudnick conjecture~\cite{MR2046966} that for any muliplicatively independent non-zero integers~$a$,~$b$ with~$\gcd(a-1, b-1)=1$, there are infinitely many~$n\in\mathbb{N}$ such that~$\gcd(a^n-1, b^n-1)=1$.
Nonetheless, using estimates for
generalized greatest common divisors due to
Bugeaud, Corvaja and Zannier~\cite{MR2130274}, it is shown in~\cite[Sec.~5]{MR3074380} that~$\srad_\Lambda=1$ for this example, and that this minimal growth rate is the case more generally when~$|\Lambda|\geqslant 3$.
\begin{theorem}[Miles~\cite{MR3074380}]
Let~$\Lambda$ be a finite set of commuting automorphisms of a connected finite dimensional compact abelian group with~$|\Lambda|\geq 3$. If every map of the form~$\widehat{\alpha}^m\widehat{\beta}^n-\widehat{\gamma}^{m+n}$ is injective for all distinct~$\alpha,\beta,\gamma\in\Lambda$ and~$m,n$ not both zero, then~$\srad_\Lambda=1$.
\end{theorem}
As well as the estimate provided by~\cite{MR2130274}, a crucial ingredient in the proof of the theorem above is the inequality
\[
|\ssyn_n(\Lambda)|
\leqslant
C\prod_{i=1}^r\prod_{v\in P_i}
\min\{|\lambda_{\alpha,i}^n-\lambda_{\gamma,i}^n|_v^{-1},|\lambda_{\beta,i}^n-\lambda_{\gamma,i}^n|_v^{-1}\},
\]
where~$C>0$ is a constant determined by~$\Lambda$,~$\alpha,\beta,\gamma\in\Lambda$ are any three distinct automorphisms, and~$P_i$ and
$\lambda_{\alpha, i},\lambda_{\beta, i},\lambda_{\gamma, i}$ are as defined for~(\ref{synch_point_counting_formula}).
Noteably, the second product on the right may be thought of as a generalized greatest common divisor.

It is also interesting to observe that \emph{synchronization point measures} may be defined in a similar way to periodic point measures, provided the maps in~$\Lambda$ are invertible. That is, we may define a~$\Lambda$-invariant measure by
\[
\smeas_{\Lambda,n}=|\ssyn_n(\Lambda)|^{-1}\sum_{x\in\ssyn_n(\Lambda)}\delta_x,
\]
where~$\delta_x$ is the Dirac mass at~$x$. The distribution properties of these synchronization point measures may be studied in a similar way to those of periodic point measures. In particular, under reasonable hypotheses,~$\smeas_{\Lambda,n}$ is seen to converge weakly to Haar measure when~$|\Lambda|=2$~\cite[Th.~5.5]{MR3074380}.  In contrast, if the conjecture of Ailon and Rudnick holds, then for the example given with~$\Lambda=\{\times 1,\times 2,\times 3\}$, the synchronization point measure~$\smeas_{\Lambda,n}$ is the Dirac mass at zero for infinitely many~$n\in\mathbb{N}$.

\section{Uniformities in mixing}

A map~$\alpha:X\to X$
preserving a measure~$\mu$
on a space~$X$ is mixing if
\[
\left\vert
\int f(x)g(\alpha^nx){\rm{d}}\mu(x)
-\int f{\rm{d}}\mu\int g{\rm{d}}\mu
\right\vert
\longrightarrow 0
\]
as~$n\to\infty$ for~$f,g\in L^2_{\mu}(X)$.
A well-known consequence of various kinds of desirable
dynamical properties for a smooth mixing map
on a compact manifold is a \emph{rate of mixing}.
This amounts to the description of a subclass
of functions~$\mathcal{C}$ (characterised by smoothness
properties, or in algebraic situations by decay of
Fourier coefficients, and so on), an appropriate
norm~$N$ on~$\mathcal{C}$, and a rate function~$\phi$
with~$\phi(n)\to 0$ as~$n\to\infty$
for which
\[
f,g\in\mathcal{C}
\Longrightarrow
\left\vert
\int f(x)g(\alpha^nx){\rm{d}}\mu(x)
-\int f{\rm{d}}\mu\int g{\rm{d}}\mu
\right\vert
\le
N(f)N(g)\phi(n)
\]
for all~$n\in\mathbb{N}$.
For a mixing~$\mathbb{Z}^d$-action by automorphisms of a compact
group several issues arise.
\begin{enumerate}
\item Describing smooth or H{\"o}lder classes of functions
is straightforward on the torus (or on nilmanifolds),
and in particular
there are convenient descriptions in
terms of Fourier coefficients. On different
groups (an infinite product of solenoids for example)
it is less clear how to do this. This already arises
for~$d=1$.
\item For~$d>1$ we have potentially very different properties:
there may be a rate of mixing in any different direction
but what can be said about possible uniformities as the
direction changes?
\item The classical Rokhlin problem, asking if mixing
forces mixing of all orders for a single measure-preserving
transformation, takes on a different shape for measure-preserving~$\mathbb{Z}^d$-actions.
On the one hand it is much simpler to answer: Ledrappier's
example shows that higher-order mixing is not
forced by mixing if~$d>1$. On the other hand, where
higher-order mixing does occur it seems to
be a by-product of subtle Diophantine
phenomena, and extending rate of mixing results to
multiple mixing brings in many difficulties.
\end{enumerate}
In this section we describe two types of result that start to
address this. There is an inevitable tension between strength
of statement and generality of the results, as will become clear.
The first result, taken from~\cite{MR2812960}
applies to all finite-dimensional
compact abelian groups but gives only
a weak description of the class of smooth
functions involved.

\begin{theorem}[Miles and Ward~\cite{MR2812960}]\label{main}
Let~$\alpha$ be a mixing entropy rank one~$\mathbb
Z^d$-action by automorphisms of a compact abelian group~$X$
satisfying the descending chain condition on.
Then there is a class of smooth functions~$\mathcal C(X)$
strictly containing the trigonometric polynomials, and a
function~$\phi=o(1)$ such that, for
any~$f,g\in\mathcal C(X)$ there is a constant~$C(f,g)$
with
\[
\left\vert\int f(x)g(\alpha^{\mathbf n}x){\rm d}\mu(x)
-\int f{\rm d}\mu\int g{\rm d}\mu\right\vert<C(f,g)\phi(\Vert\mathbf n\Vert),
\]
for all~$\mathbf{n}\in\mathbb{Z}^d$,
where~$\mu$ denotes the Haar measure on~$X$.
\end{theorem}

Just as in Theorem~\ref{first_fixed_point_uniformity_theorem},
this is essentially a Diophantine result. In order to see why
this is so, we briefly discuss what is involved in the
case of the~$\times2,\times3$
system~$(X,\alpha)$ from Example~\ref{times2_times3_example}.
Locally the group~$X$ is a product of
an open set in~$\mathbb{R}\times\mathbb{Q}_2\times\mathbb{Q}_3$,
and the action of~$\alpha^{(1,1)}$ is to multiply
by~$6$ in each direction. This map
is hyperbolic, dilating
distances in the real direction by a factor of~$6$,
in the~$2$-adic direction by a factor of~$\frac12$,
and in the~$3$-adic direction by a factor of~$\frac13$.
This hyperbolicity leads one to expect an
exponential rate of mixing for sufficiently
smooth functions under iteration of~$\alpha^{(1,1)}$.
By contrast, if~$\mathbf{n}$ is a large integer vector
chosen to lie on, or close to, one of the three
non-expansive directions then we lose (or almost
lose) hyperbolicity:
\begin{itemize}
\item if~$\mathbf{n}=(n,0)$ for some large positive~$n$, then
the corresponding automorphism multiplies by~$2^n$, which
dilates in the real direction by~$2^{n}$,
shrinks in the~$2$-adic direction by a factor~$2^{-n}$,
and acts as an isometry in the~$3$-adic direction.
\item if~$\mathbf{n}=(0,n)$ for some large positive~$n$, then
the corresponding automorphism multiplies by~$3^n$, which
dilates in the real direction by~$3^{n}$,
shrinks in the~$3$-adic direction by a factor~$3^{-n}$,
and acts as an isometry in the~$2$-adic direction.
\item if~$\mathbf{n}=(m,n)$ is chosen to lie very close to
the line~$2^x3^y=1$ for (say)~$x>0$ and~$y<0$, then
the corresponding automorphism multiplies by~$2^m3^n$.
This shrinks in the~$2$-adic direction by the
factor~$2^{-m}$, expands in the~$3$-adic direction
by the large factor~$3^n$, and in the real direction
multiplies by~$2^m3^n$, which is close to~$1$.
\end{itemize}
Thus in the first two situations we lose hyperbolicity,
and in the third we almost do. This slows the
rate of mixing in these directions, and
part of what is required to prove Theorem~\ref{main}
is checking that the rate of mixing survives
at some uniform rate across all these directions.

The second type of result is more specific about the
nature of the space but gives much stronger results
and in particular much more information about
both the class of smooth functions and the
rate of mixing. In our setting of~$\mathbb{Z}^d$-actions
by automorphisms of a compact group, the next result
applies to the case of a torus, but more generally
holds on nilmanifolds. What is remarkable about the next result
is that it not only gives an explicit rate of mixing,
it in fact gives an explicit rate of~$3$-fold mixing.

In order to state the result, which is due to Gorodnik and Spatzier~\cite{GS},
we recall that a compact nilmanifold is a quotient
space~$X=G/\Gamma$, where~$G$ is a simply connected
nilpotent group and~$\Gamma$ is a discrete and co-compact
subgroup. Write~$\aut_{\Gamma}(G)$ for the group of
continuous automorphisms of~$G$ that fix~$\Gamma$.
These automorphisms then induce maps on~$X$ that preserve
the measure~$\mu$ induced on~$X$ by the Haar measure
on~$G$, and we write~$\aut(X)$ for the group of
these maps, which are called automorphisms of the
nilmanifold~$X$. Fix a Riemannian metric on~$X$,
and finally define~$C^{\theta}(X)$ to be the
class of H{\"o}lder functions with exponent~$\theta$,
with corresponding H{\"o}lder norm~$\Vert\cdot\Vert_{\theta}$.
In the toral case,~$G=\mathbb{R}^k$,~$\Gamma=\mathbb{Z}^k$,~$\mu$ is
Lebesgue measure, and the metric is the usual metric on
the torus.

\begin{theorem}[Gorodnik and Spatzier~\cite{GS}]
Let~$X=G/\Gamma$ be a compact nilmanifold, and
let~$\alpha:\mathbb{Z}^d\longrightarrow\aut(X)$
be an action of~$\mathbb{Z}^d$ by automorphisms
of~$X$ with the property that for any~$\mathbf{n}\in\mathbb{Z}^d\setminus\{0\}$
the map~$\alpha^{\mathbf{n}}$ is ergodic with
respect to the measure~$\mu$ induced by Haar
measure on~$G$.
Then there exists a constant~$\eta(\theta)$
such that for any functions~$f,g,h\in C^{\theta}(X)$
we have a constant~$C$ with
\[
\left\vert
\int f(\alpha^{\mathbf{l}}x)
g(\alpha^{\mathbf{m}}x)
h(\alpha^{\mathbf{n}}x){\rm{d}}\mu(x)
-
\int f{\rm{d}}\mu\int g{\rm{d}}\mu\int h{\rm{d}}\mu
\right\vert
\]
\[
\le
C\left(\exp(\min\{\Vert\mathbf{l}-\mathbf{m}\Vert,
\Vert\mathbf{m}-\mathbf{n}\Vert,
\Vert\mathbf{l}-\mathbf{n}\Vert\})\right)^{\eta(\theta)}
\Vert f\Vert_{\theta}\Vert g\Vert_{\theta}\Vert h\Vert_{\theta}.
\]
\end{theorem}

\bibliographystyle{AIMS}
\bibliography{refs}

\end{document}